\documentclass[a4paper,12pt,reqno]{amsart}

\usepackage{t1enc}
\usepackage[english]{babel}
\usepackage{amsmath,amssymb,eucal,amsthm}
\usepackage{bm}
\usepackage{graphicx}
\usepackage{mathrsfs}
\usepackage{mathptmx}
\usepackage{latexsym}
\usepackage{ulem}

\newtheorem{theorem}{Theorem}
\newtheorem{lemma}[theorem]{Lemma}

\theoremstyle{definition}
\newtheorem{remark}{Remark}

\def\pp{\mathbb{P}}
\def\nn{\mathbb{N}}
\def\zz{\mathbb{Z}}

\def\cc{\mathbb{C}}
\def\qq{\mathbb{Q}}

\def\gb{\mathfrak{B}}

\def\d{{\rm d}}
\def\meas{{\rm meas}}
\def\st{{\widetilde{S}}}

\begin{document}
\hfill\texttt{\jobname.tex}\qquad\today

\bigskip
\title[Mixed joint discrete universality]
{On mixed joint discrete universality \\ for a class of zeta-functions}

\author{Roma Ka{\v c}inskait{\.e}}

\address{R. Ka{\v c}inskait{\.e} \\
Department of Mathematics and Statistics, Vytautas Magnus University, Kaunas, Vileikos 8, LT-44404, Lithuania\\
Department of Informatics and Ma\-the\-ma\-tics, {\v S}iauliai University, Vilniaus~141, LT-76353 {\v S}iauliai, Lithuania}
\email{r.kacinskaite@if.vdu.lt, r.kacinskaite@fm.su.lt}

\author{Kohji Matsumoto}

\address{K. Matsumoto, Graduate School of Mathematics, Nagoya University, Chikusa-ku,
Nagoya 464-8602, Japan}
\email{kohjimat@math.nagoya-u.ac.jp}
\date{}

\maketitle

\begin{abstract}
We prove a mixed joint discrete universality theorem for a Matsumoto zeta-function $\varphi(s)$
(belonging to the Steuding subclass) and a
periodic Hurwitz zeta-function $\zeta(s,\alpha;\gb)$.
For this purpose, certain independence condition for the parameter $\alpha$ and the minimal step
of discrete shifts of these functions is assumed.
This paper is a continuation of authors' works \cite{RK-KM-15} and \cite{RK-KM-bams}.
\end{abstract}



\section{Introduction}\label{sect-1}

In analytic number theory, the problem of so-called mixed joint universality in Vo\-ro\-nin's sense is a very interesting problem since it solves a problem on simultaneous approximation of certain tuples of analytic functions by shifts of tuples consisting of zeta-functions having an Euler product expansion over the set of primes and other zeta-functions without such product. For such type of universality, a very important role is played by the parameters which occur in the definitions of zeta-functions.

The first result on mixed joint universality was obtained by H.~Mishou in \cite{HM-07}. He proved that the Riemann zeta-function $\zeta(s)$ and the Hurwitz zeta-function $\zeta(s,\alpha)$ with transcendental parameter $\alpha$ are jointly universal.

Let $\mathbb{P}$, $\mathbb{N}$, $\mathbb{N}_0$, $\mathbb{Z}$, $\mathbb{Q}$, $\mathbb{R}$ and $\mathbb{C}$ be the sets of all primes, positive integers, non-negative integers, integers, rational numbers, real numbers and complex numbers, respectively. Denote by $s=\sigma+it$ a complex variable.
Recall that the functions $\zeta(s)$ and $\zeta(s,\alpha)$, $0<\alpha \leq 1$, for $\sigma>1$, are defined by
$$
\zeta(s)=\sum_{m=1}^{\infty}\frac{1}{m^s}=\prod_{p \in \pp}\bigg(1-\frac{1}{p^s}\bigg)^{-1}\quad \text{and}\quad
\zeta(s,\alpha)=\sum_{m=0}^{\infty}\frac{1}{(m+\alpha)^s},
$$
respectively.
Both of them are analytically continued to the whole complex plane, except for a simple pole at the point $s=1$ with residue 1. Note that the Riemann zeta-function has the Euler product expansion, while in general the Hurwitz zeta-function does not have (except the cases $\alpha=\frac{1}{2}, \, 1$).

For Mishou's result and further statements, we introduce some notations.
Let $D(a,b)=\{s\in\cc:a<\sigma<b\}$ for any $a<b$.
For every compact set $K\subset\cc$, denote by $H^c(K)$ the set of all
$\cc$-valued continuous functions defined on $K$ and holomorphic in the interior of
$K$. By $H_0^c(K)$ we denote  the subset of $H^c(K)$, consisting of all elements which are
non-vanishing on $K$.

\begin{theorem}[\cite{HM-07}]\label{Mishou-2007}
Suppose that $\alpha$ is a transcendental number. Let $K_1$ and $K_2$ be compact subsets of the strip $D(\frac{1}{2},1)$ with connected complements. Suppose that $f_1(s) \in H_0^c(K_1)$ and $f_2(s) \in H^c(K_2)$. Then, for every $\varepsilon>0$,
$$
\liminf\limits_{T\to \infty}\frac{1}{T}{\meas}\bigg\{\tau \in [0,T]: \sup\limits_{s \in K_1}|\zeta(s+i\tau)-f_1(s)|<\varepsilon,\sup\limits_{s \in K_2}|\zeta(s+i\tau,\alpha)-f_2(s)|<\varepsilon\bigg\}>0.
$$
\end{theorem}
\noindent Here, as usual, $\meas\{A\}$ denotes the Lebesgue measure of the measurable set $A \subset {\mathbb{R}}$.

Note that J. Sander and J. Steuding \cite{JS-JS-06} proved the same type of universality, but for
rational $\alpha$, by a quite different method.

In \cite{RK-KM-15}, we consider the mixed joint universality property for a wide  class of zeta-functions consisting of Matsumoto zeta-functions $\varphi(s)$ belonging to the Steuding class
${\st}$ and periodic Hurwitz zeta-functions $\zeta(s,\alpha;\gb)$.

Recall the definition of the polynomial Euler products ${\widetilde{\varphi}}(s)$ or so-called Matsumoto zeta-functions.   (Remark: The function ${\widetilde \varphi}(s)$ was introduced by the second author in \cite{KM-90}.)
For $m \in \nn$, let $g(m)$ be a positive integer and $p_m$ the $m$th prime number. Moreover, let $a_m^{(j)}\in \cc$, and $f(j,m)\in \mathbb{N}$, $1\leq j \leq g(m)$.
The function ${\widetilde{\varphi}}(s)$ is defined by the polynomial Euler product
\begin{equation}\label{rk-eq-1}
{\widetilde\varphi}(s)=\prod_{m=1}^{\infty}\prod_{j=1}^{g(m)}\left(1-a_m^{(j)}p_m^{-sf(j,m)}\right)^{-1}.
\end{equation}
It is assumed that
\begin{equation}\label{rk-eq-2}
g(m)\leq C_1p_m^\alpha \quad \text{and} \quad  |a_m^{(j)}|\leq p_m^\beta
\end{equation}
with a positive constant $C_1$ and non-negative constants $\alpha$ and $\beta$. In view of \eqref{rk-eq-2},
the function $\widetilde{\varphi}(s)$ converges absolutely for $\sigma>\alpha+\beta+1$, and hence, in this region, it can be given by the absolutely convergent Dirichlet series
\begin{align}\label{rk-eq-2.5}
{\widetilde \varphi}(s)=\sum_{k=1}^{\infty}\frac{{\widetilde c}_k}{k^s}.
\end{align}
The shifted function ${\varphi}(s)$ is given by
\begin{equation}\label{rk-eq-3}
\varphi(s)=\sum_{k=1}^{\infty}\frac{{\widetilde c}_k}{k^{s+\alpha+\beta}}=\sum_{k=1}^{\infty}\frac{c_k}{k^s}
\end{equation}
with $c_k=k^{-\alpha-\beta}{\widetilde c}_k$.
For $\sigma>1$, the last series in \eqref{rk-eq-3} converges absolutely too.

Also, suppose that, for the function $\varphi(s)$, the following assumptions hold (for the
details, see \cite{KM-90}):
\begin{itemize}
  \item[(a)] $\varphi(s)$ can be continued meromorphically to
$\sigma\geq\sigma_0$, where $\frac{1}{2}\leq\sigma_0<1$, and all poles in this
region are included in a compact set which has no intersection with the line $\sigma=\sigma_0$,
  \item[(b)] $\varphi(\sigma+it)=O(|t|^{C_2})$ for $\sigma\geq\sigma_0$, where $C_2$ is a positive constant,
  \item[(c)] the mean-value estimate
\begin{equation}\label{rk-eq-2-5}
\int_0^T|\varphi(\sigma_0+it)|^2 dt=O(T).
\end{equation}
\end{itemize}

It is possible to discuss functional limit theorems for Matsumoto zeta-functions (see Section
\ref{sect-2} below), but this framework is too wide to consider the universality property.
To investigate the universality, we introduce the Steuding subclass $\st$, for which the following
slightly more restrictive conditions are required. We say that the function $\varphi(s)$ belongs to the class $\st$, if the following conditions are fulfilled:
\begin{itemize}
  \item[(i)] there exists a Dirichlet series expansion
  $$
  \varphi(s)=\sum_{m=1}^{\infty}\frac{a(m)}{m^s}
  $$
  with $a(m)=O(m^\varepsilon)$ for every $\varepsilon>0$;
  \item[(ii)] there exists $\sigma_\varphi<1$ such that $\varphi(s)$ can be meromorphically continued to the half-plane $\sigma>\sigma_\varphi$, and holomorphic except for at most a pole at $s=1$;
  \item[(iii)] there exists a constant $c \geq 0$ such that
  $$
  \varphi(\sigma+it)=O(|t|^{c+\varepsilon})
  $$
  for every fixed $\sigma>\sigma_\varphi$ and $\varepsilon>0$;
  \item[(iv)] there exists the Euler product expansion over prime numbers, i.e.,
  $$
  \varphi(s)=\prod_{p \in \mathbb{P}}\prod_{j=1}^{l}\left(1-\frac{a_j(p)}{p^s}\right)^{-1};
  $$
  \item[(v)] there exists a constant $\kappa>0$ such that
  $$
  \lim_{x \to \infty}\frac{1}{\pi(x)}\sum_{p \leq x}|a(p)|^2=\kappa,
  $$
  where $\pi(x)$ denotes the number of primes up to $x$, i.e.,  $p \leq x$.
\end{itemize}

For $\varphi\in{\widetilde S}$, let $\sigma^*$ be the infimum of all $\sigma_1$ for which
$$
\frac{1}{2T}\int_{-T}^T |\varphi(\sigma+it)|^2 dt\sim\sum_{m=1}^{\infty}
\frac{|a(m)|^2}{m^{2\sigma}}
$$
holds for every $\sigma\geq\sigma_1$. Then it is known that $\frac{1}{2}\leq\sigma^*<1$. (Remark: The class $\st$ was introduced by J.~Steuding in  \cite{JSt-07}.)

Now we recall the definition of the periodic Hurwitz zeta-function $\zeta(s,\alpha;\gb)$ with a fixed parameter $\alpha$, $0<\alpha\leq 1$.  (Remark: The function $\zeta(s,\alpha;\gb)$ was introduced by A.~Javtokas and A.~Laurin\v cikas in \cite{AJ-AL-06}.)   Let ${\mathfrak{B}}=\{b_m: m \in \nn_0\}$ be a periodic sequence of complex numbers (not all zero) with minimal period $k \in \nn$.
For $\sigma>1$, the function $\zeta(s,\alpha;\gb)$ is defined by
$$
\zeta(s,\alpha;\gb)=\sum_{m=0}^{\infty}\frac{b_m}{(m+\alpha)^s}.
$$
It is known that
\begin{eqnarray}\label{rk-eq-4*}
\zeta(s,\alpha;\gb)=\frac{1}{k^s}\sum_{l=0}^{k-1}b_l\zeta\bigg(s,\frac{l+\alpha}{k}\bigg), \quad \sigma>1.
\end{eqnarray}
The last equality gives an analytic continuation of the function $\zeta(s,\alpha;\gb)$ to the whole complex plane, except for a possible simple pole at the point $s=1$ with residue
$$
b:=\frac{1}{k}\sum_{l=0}^{k-1}b_l.
$$
If $b=0$, then $\zeta(s,\alpha;\gb)$  is an entire function.

In \cite{RK-KM-15}, we prove the mixed joint universality property of the functions $\varphi(s)$ and $\zeta(s,\alpha;\gb)$.

\begin{theorem}[\cite{RK-KM-15}]\label{rk-th-1}
Suppose that $\varphi(s)\in {\widetilde S}$, and $\alpha$ is a transcendental number. Let
$K_1$ be a compact subset of $D(\sigma^*,1)$, and $K_2$ be a compact subset of
$D(\frac{1}{2},1)$, both with connected complements. Suppose that $f_1\in H_0^c(K_1)$
and $f_2\in H^c(K_2)$.   Then, for every $\varepsilon>0$,
\begin{align*}
\liminf\limits_{T \to \infty}\frac{1}{T}\meas\bigg\{\tau\in [0,T]: &\; \sup\limits_{s \in K_1}|\varphi(s+i\tau)-f_1(s)|<\varepsilon, \\ &\; \sup\limits_{s\in K_2}|\zeta(s+i\tau,\alpha;\gb)-f_2(s)|<\varepsilon\bigg\}>0.
\end{align*}
\end{theorem}

In \cite{RK-KM-bams}, we obtain a generalization of Theorem \ref{rk-th-1}, in which several
periodic Hurwitz zeta-functions are involved.

More interesting and convenient in practical applications is so-called discrete universality of zeta-functions (for example, see \cite{KB-NK-HR-91}). This pushes us to extend our investigations of mixed joint universality for a class of zeta-functions to the discrete case. Recall that, in this case, the pair of analytic functions is approximated by discrete shifts of tuple $\big(\varphi(s+ikh), \zeta(s+ikh,\alpha;\gb)\big)$, $k \in \nn_0$, where $h>0$ is the minimal step of given arithmetical progression.

The aim of this paper is to prove a mixed joint discrete universality theorem for the collection of the functions $(\varphi(s), \zeta(s,\alpha;\gb))$, i.e., the discrete version of Theorem~\ref{rk-th-1}.

For $h>0$, let
$$
L(\pp,\alpha,h)=\bigg\{(\log p: p \in \pp), (\log(m+\alpha): m \in \nn_0),\frac{2 \pi}{h}\bigg\}.
$$

\begin{theorem}\label{rk-th-2}
Let $\varphi(s)\in {\widetilde S}$, $K_1$, $K_2$, $f_1(s)$ and $f_2(s)$ satisfy the conditions as in Theorem~\ref{rk-th-1}. Suppose that the set $L(\pp,\alpha,h)$ is linearly independent over $\qq$.
Then, for every $\varepsilon>0$,
\begin{eqnarray*}
\liminf\limits_{N \to \infty}
\frac{1}{N+1}
\#
\bigg\{0\leq k \leq N:
 && \sup\limits_{s \in K_1}|\varphi(s+ikh)-f_1(s)|<\varepsilon, \\ &&  \sup\limits_{s\in K_2}|\zeta(s+ikh,\alpha;\gb)-f_2(s)|<\varepsilon\bigg\}>0.
\end{eqnarray*}
\end{theorem}

\begin{remark}\label{rem-1}
A typical situation when $L(\pp,\alpha,h)$ is linearly independent is the case when $\alpha$ and
$\exp\big\{\frac{2\pi}{h}\big\}$ are algebraically independent over $\qq$.   The proof of this fact is
given in \cite{EB-AL-15rj}.
\end{remark}

Now recall some known facts of discrete universality which directly connect with objects under our interests.

Discrete universality property for the Matsumoto zeta-function under the
condition that $\exp \{\frac{2 \pi k}{h}\}$ is irrational for every non-zero integer $k$ was
obtained by the first author in \cite{RK-02}. While the discrete universality of the periodic Hurwitz zeta-functions was proved by A.~Laurin\v cikas and R.~Macaitien\.e in \cite{AL-RM-09}.

Also, some results on discrete analogue of mixed universality are known. The first attempt in this direction was done by the first author in \cite{RK-09}, under the assumption that $\alpha$ is transcendental and $\exp\big\{\frac{2\pi}{h}\big\}$ is rational.    Unfortunately the proof in \cite{RK-09} is incomplete, as mentioned by A.~Laurin\v cikas in 2014 (see \cite{EB-AL-15lmj}).
However the argument in \cite{RK-09} gives a correct proof for the modified $L$-functions where
all Euler factors corresponding to primes in the set of all prime numbers appearing as a prime factor of $a$ or $b$ such that  $\frac{a}{b}=\exp\big\{\frac{2 \pi}{h}\big\} \in \qq$, $a,b \in \zz$, $(a,b)=1$, are removed; see Section \ref{sect-5}.

In \cite{EB-AL-15rj} and \cite{EB-AL-15lmj}, E.~Buivydas and A.~Laurin\v cikas proved the joint mixed discrete universality for the Riemann zeta-function  $\zeta(s)$ and Hurwitz zeta-function $\zeta(s,\alpha)$. The first result \cite{EB-AL-15rj} deals with the case when the mi\-ni\-mal steps of
arithmetical progressions $h$ for both functions are common, while in the second paper \cite{EB-AL-15lmj}, for $\zeta
(s)$ and $\zeta(s,\alpha)$, the mi\-ni\-mal steps $h_1$ and $h_2$ are different from each other.

It is the purpose of the present paper to give the proof of
the joint mixed discrete universality theorem (Theorem \ref{rk-th-2}) for
$(\varphi(s),\zeta(s,\alpha;\gb))$, which
generalizes the result from \cite{EB-AL-15rj}, and to clarify the situation in \cite{RK-09}.

\section{A joint mixed discrete limit theorem}\label{sect-2}

The proof of Theorem~\ref{rk-th-2} is based on a joint mixed discrete limit theorem in the sense of weakly convergent probability
measures in the space of analytic functions for the Matsumoto zeta-functions $\varphi(s)$ and the periodic Hurwitz zeta-function $\zeta(s,\alpha;{\mathfrak{B}})$ (for the detailed
expositions of this method, see \cite{AL-96}, \cite{JSt-07}), which we prove in this section,
using the linear independence of the set $L(\pp,\alpha,h)$.
In this section, $\varphi(s)$ denotes any general Matsumoto zeta-function.

For further statements, we start with some notations and definitions.

For a set $S$, denote by $\mathcal{B}(S)$ the set of all Borel subset of $S$.
Let $\gamma=\{s \in \cc: |s|=1\}$. Define
$$
\Omega_1=\prod_{p\in\pp} \gamma_p \quad\text{and} \quad \Omega_2=\prod_{m=0}^{\infty}\gamma_m,
$$
where $\gamma_p=\gamma$ for all $p\in\pp$, and $\gamma_m=\gamma$ for all $m \in \nn_0$.
By the Tikhonov theorem (see \cite{JLK-55}), the tori $\Omega_1$ and $\Omega_2$ with the product topology and the pointwise multiplication are compact topological groups. Then
$$
\Omega:=\Omega_1 \times \Omega_2
$$
is a compact topological Abelian group too, and we obtain the probability space $(\Omega, {\mathcal{B}}(\Omega),$ $m_H)$. Here $m_H=m_{1H}\times m_{2H}$ with the probability Haar measures $m_{1H}$ and $m_{2H}$ defined on the spaces  $(\Omega_1,{\mathcal B}(\Omega_1))$ and $(\Omega_2,{\mathcal B}(\Omega_2))$, respectively.

Let $\omega_1(p)$ stand for the projection of
$\omega_1\in\Omega_1$ to the coordinate space $\gamma_p$, $p \in \pp$, and, for every $m\in\nn$, we put
$$
\omega_1(m)=\prod_{j=1}^r \omega_1(p_j)^{l_j},
$$
where, by factorizing of $m$ into the primes, $m=p_1^{l_1}\cdots p_r^{l_r}$. Let $\omega_2(m)$ denotes the projection
of $\omega_2\in\Omega_2$ to the coordinate space $\gamma_m$, $m \in \nn_0$. Define $\omega=(\omega_1,\omega_2)$ for elements of $\Omega$.
For any open subregion $G$ in the complex plane,
let $H(G)$ be the space of analytic functions on $G$ equipped with the topology of uniform convergence in compacta.


The function $\varphi(s)$ has only finitely many poles by the condition (a). Denote those
poles by $s_1(\varphi),\ldots,s_l(\varphi)$, and define
$$
D_{\varphi}=\{s: \;\sigma>\sigma_0,\; \sigma\neq\Re s_j(\varphi), \;1\leq j\leq l\}.
$$
Then $\varphi(s)$ and its vertical shift $\varphi(s+ikh)$ are holomorphic in
$D_{\varphi}$.
The function $\zeta(s,\alpha;\gb)$ can be written as a linear combination of
Hurwitz zeta-functions \eqref{rk-eq-4*}, therefore it is entire, or has a simple pole at $s=1$.
Therefore $\zeta(s,\alpha;\gb)$ and its vertical shift $\zeta(s+ikh,\alpha;\gb)$
are holomorphic in
$$
D_{\zeta}=
\begin{cases}
\big\{s \in \cc: \; \sigma>\frac{1}{2}\big\} & \text{if}\quad \zeta(s,\alpha;\gb)\;\; \text{is entire},\cr
\big\{s:\;\sigma>\frac{1}{2},\sigma\neq 1\big\} & \text{if}\quad s=1\;\; \text{is a pole of}\;\; \zeta(s,\alpha;\gb).
\end{cases}
$$

Now, in view of the definitions of $D_{\varphi}$ and $D_{\zeta}$, let $D_1$ and $D_2$ be two open subsets of $D_{\varphi}$ and $D_{\zeta}$, respectively.
Let $\underline{H}=H(D_1)\times H(D_2)$.
On $(\Omega,{\mathcal B}(\Omega),m_H)$, define the $\underline{H}$-valued random element
$\underline{Z}(\underline{s},\omega)$ by the formula
$$
\underline{Z}(\underline{s},\omega)=\big(\varphi(s_1,\omega_1), \zeta(s_2,\alpha,\omega_2;\gb)\big),
$$
where $\underline{s}=(s_1,s_2) \in D_1\times D_2$,
\begin{equation}\label{rk-eq-3*}
\varphi(s_1,\omega_1)=\sum_{k=1}^{\infty}\frac{c_k \omega_1(k)}{k^{s_1}}
\end{equation}
and
\begin{equation}\label{rk-eq-4}
\zeta(s_2,\alpha,\omega_2;\gb)=\sum_{m=0}^{\infty}
\frac{b_m\omega_2(m)}{(m+\alpha)^{s_2}},
\end{equation}
respectively.

Denote by $P_{\underline{Z}}$
the distribution of $\underline{Z}(\underline{s},\omega)$ as an
$\underline{H}$-valued random element, i.e.,
$$
P_{\underline{Z}}(A)=m_H\{\omega\in\Omega :\;\underline{Z}(\underline{s},\omega)\in A\},
\quad A\in\mathcal{B}(\underline{H}).
$$

Let $N>0$. Define the probability measure $P_N$ on $\underline{H}$ by the formula
$$
P_{N}(A)=\frac{1}{N+1}\#\big\{0 \leq k \leq N: \;
\underline{Z}(\underline{s}+ikh)
\in A\big\}, \quad A\in\mathcal{B}(\underline{H}),
$$
where $\underline{s}+ikh=(s_1+ikh,s_2+ikh)$ with
$s_1\in D_1$, $s_2\in D_2$, and
$$
\underline{Z}(\underline{s})=
(\varphi(s_1),\zeta(s_2,\alpha;\gb)).
$$

In the course of the proof of Theorem~\ref{rk-th-2}, the first main goal is the following mixed joint discrete limit theorem.

\begin{theorem}\label{rk-th-3}
Suppose that the set $L(\pp,\alpha,h)$ is linearly independent over $\qq$. Then the probability measure
$P_N$ converges weakly to $P_{\underline{Z}}$ as $N \to \infty$.
\end{theorem}

We will omit some details of the proof, because the proof follows in the standard way (see, for example, the proof of Theorem~7 of \cite{EB-AL-15rj}).     However, though the following lemma, a mixed joint discrete limit theorem on the torus $\Omega$,
is exactly the same as Lemma 1 of \cite{EB-AL-15rj}, we reproduce the detailed proof, since this result plays a crucial role, and from the proof we can see why the linear independence of
$L(\pp,\alpha,h)$ is necessary.

Define
$$
Q_N(A):=\frac{1}{N+1}\# \bigg\{0 \leq k \leq N: \big(\big(p^{-ikh}: p \in \pp\big), \big((m+\alpha)^{-ikh}: m \in \nn_0\big)\big)\in A\bigg\}, \quad A \in {\mathcal B}(\Omega).
$$

\begin{lemma}[\cite{EB-AL-15rj}]\label{rk-lem-1}
Suppose that the set $L(\pp,\alpha,h)$ satisfies the condition of Theorem~\ref{rk-th-2}. Then $Q_N$ converges weakly to the Haar measure $m_H$ as $N \to \infty$.
\end{lemma}

\begin{proof} For the proof of Lemma~\ref{rk-lem-1}, we use the Fourier transformation method (for the details, see \cite{AL-96}).   The dual group of $\Omega$ is isomorphic to the group
$$
G:=\bigg(\bigoplus\limits_{p \in \pp}\zz_p\bigg)\bigoplus\bigg(\bigoplus\limits_{m \in \nn_0}\zz_m\bigg)
$$
with $\zz_p=\zz$ for all $p \in \pp$ and $\zz_m=\zz$ for all $m \in \nn_0$.
The element of $G$ is written as
$(\underline{k},\underline{l})=((k_p: p \in \pp), (l_m: m \in \nn_0)),$
where only finite number of integers $k_p$ and $l_m$ are non-zero, and acts on $\Omega$ by
$$
(\omega_1,\omega_2)\to (\omega_1^{\underline{k}},\omega_2^{\underline{l}})=\prod_{p \in \pp}\omega_1^{k_p}(p)\prod_{m \in \nn_0}\omega_2^{l_m}(m).
$$
Let $g_N(\underline{k}, \underline{l})$, for $(\underline{k}, \underline{l})\in G$, be the Fourier transform of the measure $Q_N(A)$.
Then we have
$$
g_N(\underline{k},\underline{l})=\int_{\Omega}\bigg(\prod_{p \in \pp}\omega_1^{k_p}(p)\prod_{m \in \nn_0}\omega_2^{l_m}(m)\bigg)\d Q_N.
$$
Thus, from the definition of $Q_N(A)$,
\begin{eqnarray}\label{rk-eq-5}
g_N(\underline{k},\underline{l})&=&
\frac{1}{N+1}\sum_{k=0}^{N}\prod_{p \in \pp}p^{-ikk_ph}\prod_{m \in \nn_0}(m+\alpha)^{-ikl_mh}\cr
&=&
\frac{1}{N+1}\sum_{k=0}^{N}
\exp\bigg\{-ikh\bigg(\sum_{p \in \pp}k_p\log p+\sum_{m \in \nn_0}l_m \log(m+\alpha)\bigg)\bigg\}.
\end{eqnarray}
By the assumption of the lemma, the set $L(\pp,\alpha,h)$ is linearly independent over $\qq$. Then the set $\{(\log p: p \in \pp),(\log(m+\alpha):m \in \nn_0)\}$ is linearly independent over $\qq$, and
$$
\sum_{p \in \pp}k_p\log p+\sum_{m \in \nn_0}l_m \log(m+\alpha)=0
$$
if and only if $\underline{k}=\underline{0}$ and $\underline{l}=\underline{0}$. Moreover, if $(\underline{k},\underline{l})\not =(\underline{0},\underline{0})$,
\begin{equation}\label{rk-eq-6}
\exp\bigg\{-ih\bigg(\sum_{p \in \pp}k_p\log p+\sum_{m \in \nn_0}l_m \log(m+\alpha)\bigg)\bigg\}\not =1.
\end{equation}
In fact, if \eqref{rk-eq-6} is false, then
\begin{align}\label{rk-eq-6-2}
\sum_{p \in \pp}k_p\log p+\sum_{m \in \nn_0}l_m \log(m+\alpha) =\frac{2 \pi a}{h}
\end{align}
with some $a\in \zz\setminus\{0\}$. But this contradicts to the linear independence of the set $L(\pp,\alpha,h)$.
Therefore, from \eqref{rk-eq-5} and \eqref{rk-eq-6}, we find that
\begin{eqnarray*}
&&g_N(\underline{k},\underline{l})\cr
&& \quad =\begin{cases}
1, & \text{if} \quad (\underline{k},\underline{l})=(\underline{0},\underline{0}),\cr
\frac{1-\exp\big\{-i(N+1)h\big(\sum_{p \in \pp}k_p\log p+ \sum_{m\in \nn_0}l_m\log(m+\alpha)\big)\big\}}
{(N+1)\big(1-\exp\big\{-ih\big(\sum_{p \in \pp}k_p\log p+\sum_{m \in \nn_0}l_m \log(m+\alpha)\big)\big\}\big)}, & \text{if} \quad (\underline{k},\underline{l})\not =(\underline{0},\underline{0}).
\end{cases}
\end{eqnarray*}
Hence,
$$
\lim\limits_{N \to \infty}g_N(\underline{k},\underline{l})=
\begin{cases}
1, &\text{if} \quad (\underline{k},\underline{l})=(\underline{0},\underline{0}),\cr
0, &\text{otherwise}.
\end{cases}
$$

By a continuity theorem for probability measures on compact groups (see \cite{HH-77}), we obtain the statement of the lemma, i.e., that $Q_N(A)$ converges weakly to $m_H$ as $N \to \infty$.
\end{proof}

Now, using Lemma~\ref{rk-lem-1}, we may prove a joint mixed discrete limit theorem for absolutely convergent Dirichlet series.

Let, for fixed $\widehat{\sigma}>\frac{1}{2}$,
$$
v_1(m,n)=\exp\bigg\{-\bigg(\frac{m}{n}\bigg)^{\widehat{\sigma}}\bigg\}, \quad m,n \in \nn,
$$
and
$$
v_2(m,n,\alpha)=\exp\bigg\{-\bigg(\frac{m+\alpha}{n+\alpha}\bigg)^{\widehat{\sigma}}\bigg\},
\quad m \in \nn_0, \quad n \in \nn.
$$
Define the series
\begin{eqnarray*}
\varphi_n(s)&=&\sum_{m=1}^{\infty}\frac{c_m v_1(m,n)}{m^s},\cr
\zeta_n(s,\alpha;\gb)&=&\sum_{m=0}^{\infty}\frac{b_m v_2(m,n,\alpha)}{(m+\alpha)^s},
\end{eqnarray*}
and, for $\widehat{\omega}:=\big(\widehat{\omega}_1,\widehat{\omega}_2\big) \in \Omega$,
\begin{eqnarray*}
\varphi_n(s,\widehat{\omega}_1)&=&\sum_{m=1}^{\infty}\frac{\widehat{\omega}_1(m)c_m v_1(m,n)}{m^s},\cr
\zeta_n(s,\alpha,\widehat{\omega}_2;\gb)&=&\sum_{m=0}^{\infty}\frac{\widehat{\omega}_2(m)b_m v_2(m,n,\alpha)}{(m+\alpha)^s},
\end{eqnarray*}
respectively.
These series are absolutely convergent for $\sigma>\frac{1}{2}$.

For brevity, denote
$$
{\underline{Z}}_n(\underline{s})=(\varphi_n(s_1),\zeta_n(s_2,\alpha;\gb))
$$
and
$$
{\underline{Z}}_n(\underline{s}, \widehat{\omega})=(\varphi_n(s_1,{\widehat{\omega}}_1),\zeta_n(s_2,\alpha,\widehat{\omega}_2;\gb)).
$$

Now, on the space $(\underline{H},{\mathcal B}({\underline{H}}))$, we consider the weak convergence of the measures
$$
P_{N,n}(A)=\frac{1}{N+1}\# \bigg\{0\leq k \leq N: {\underline{Z}}_n(\underline{s}+ikh)\in A \bigg\},
$$
and, for $\widehat{\omega }\in \Omega$,
$$
\widehat{P}_{N,n}(A)=\frac{1}{N+1}\# \bigg\{0\leq k \leq N: {\underline{Z}}_n(\underline{s}+ikh,\widehat{\omega})\in A \bigg\}.
$$

\begin{lemma}\label{rk-lem-2}
Suppose that the set $L(\pp, \alpha,h)$ is linearly independent over $\qq$. Then, on
$(\underline{H}, {\mathcal B}(\underline{H}))$, there exists a probability measure $P_n$ such that the measures $P_{N,n}$ and ${\widehat{P}}_{N,n}$ both converge weakly to $P_n$ as $N \to \infty$.
\end{lemma}

\begin{proof}
The proof of the lemma follows analogous to Lemma~2 from \cite{EB-AL-15rj}.
\end{proof}

The next step of the proof is to approximate the tuple $({\underline Z}(\underline{s}),{\underline{Z}}(\underline{s},\widehat{\omega}))$ by the tuple $\big({\underline Z}_n(\underline{s}), {\underline Z}_n(\underline{s},\widehat{\omega})\big)$. For this purpose, we will use the metric on the space $\underline{H}$.
For any open region $G$,
it is known (see \cite{JBC-78} or \cite{AL-96}) that there exists a sequence of compact sets $\{K_l: l\in \nn\}\subset {G}$ satisfying conditions:
\begin{enumerate}
  \item $G=\bigcup\limits_{l=1}^\infty K_l$,
  \item $K_l \subset K_{l+1}$ for any $l \in \nn$,
  \item if $K$ is a compact set, then $K \subset K_l$ for some $l \in \nn$.
\end{enumerate}
For functions $g_1,g_2 \in H(G)$, define a metric $\varrho_G$ by the formula
$$
\varrho_G(g_1,g_2)=\sum_{l=1}^{\infty}\frac{1}{2^l}\frac{\sup_{s \in K_l}|g_1(s)-g_2(s)|}{1+\sup_{s \in K_l}|g_1(s)-g_2(s)|}
$$
which induces the  topology of uniform convergence on compacta.
Put  $\varrho_1=\varrho_{D_1}$ and $\varrho_2 =\varrho_{D_2}$. Define, for
$\underline{g}_1=(g_{11},g_{21})$ and $\underline{g}_2=(g_{12},g_{22})$ from $\underline{H}$,
$$
{\underline \varrho}(\underline{g}_1,\underline{g}_2)=\max\big\{\varrho_1(g_{11},g_{12}),\varrho_2(g_{21},g_{22})\big\}.
$$
In such a way, we obtain a metric on the space $\underline{H}$ including its topology.

\begin{lemma}\label{rk-lem-3}
Suppose that the set $L(\pp,\alpha,h)$ is linearly independent over $\qq$. The equalities
\begin{align}\label{rk-lem-3-1}
\lim\limits_{n \to \infty}
\limsup\limits_{N \to \infty}
\frac{1}{N+1}
\sum_{k=0}^{N}
\underline{\varrho} \big({\underline Z}(\underline{{s}}+ikh),{\underline Z}_n(\underline{s}+ikh)\big)=0
\end{align}
and, for almost all $\omega \in \Omega$,
\begin{align}\label{rk-lem-3-2}
\lim\limits_{n \to \infty}
\limsup\limits_{N \to \infty}
\frac{1}{N+1}
\sum_{k=0}^{N}
\underline{\varrho} \big({\underline Z}(\underline{s}+ikh,\omega),{\underline Z}_n(\underline{s}+ikh,\omega)\big)=0
\end{align}
hold.
\end{lemma}

\begin{proof}
This can be shown in a way similar to the proofs of Lemmas~3 and 4 of \cite{EB-AL-15rj},
respectively.    The main body of the argument in \cite{EB-AL-15rj}, based on an application
of Gallagher's lemma, is going back to the proof of Theorem 4.1 of \cite{AL-RM-09}.
We just indicate some different points from the proof in \cite{EB-AL-15rj} and \cite{AL-RM-09}.

The starting point of the proof of \eqref{rk-lem-3-1} is the integral expression
\begin{align}\label{intexp1}
\varphi_n(s)=\frac{1}{2\pi i}\int_{a-i\infty}^{a+i\infty}\varphi(s+z)l_n(z)\frac{dz}{z}
\end{align}
and
\begin{align}\label{intexp2}
\zeta_n(s,\alpha;\gb)=\frac{1}{2\pi i}\int_{a-i\infty}^{a+i\infty}\zeta(s+z,\alpha;\gb)
l_n(z,\alpha)\frac{dz}{z},
\end{align}
where $a>\frac{1}{2}$, and
$$
l_n(z)=\frac{z}{a}\Gamma\left(\frac{z}{a}\right)n^z
\quad \text{and} \quad l_n(z,\alpha)=\frac{z}{a}\Gamma\left(\frac{z}{a}\right)(n+\alpha)^z,
$$
respectively.
We shift the paths to the left and apply the residue calculus.
The case \eqref{intexp2} is discussed in \cite{AL-RM-09}, where the path is moved to
$\Re z=b-\sigma$ with $\frac{1}{2}<b<1$ and $\sigma>b$.   In this case, the relevant poles are only
$z=0$ and $z=1-s$.
As for \eqref{intexp1}, we shift the path to $\Re z=\sigma_0+\delta_0-\sigma$, where
$\delta_0$ is a small positive number such that $\varphi(s)$ is holomorphic in the strip
$\sigma_0\leq\Re s\leq \sigma_0+\delta_0$.
We encounter all the poles $z=s_j(\varphi)-s$, $1\leq j\leq l$, so we have to consider all
the residues coming from those poles.   But they can be handled by the same method as
described in the proof of Theorem 4.1 of \cite{AL-RM-09}.

To complete the proof of \eqref{rk-lem-3-1}, it is also necessary to show the discrete mean square
estimate
\begin{align}\label{disc-mean}
\sum_{k=0}^N|\varphi(\sigma_0+\delta_0+it+ikh)|^2 \ll N(1+|t|).
\end{align}
This is an analogue of Lemma 4.3 of \cite{AL-RM-09}, and can be obtained similarly from
\eqref{rk-eq-2-5} and Gallagher's lemma (Lemma 1.4 of \cite{HLM-71}).

As for the proof of \eqref{rk-lem-3-2}, we need the ``random'' version of \eqref{rk-eq-2-5}, that is
\begin{align}\label{rk-eq-2-5-random}
\int_0^T|\varphi(\sigma+it,\omega_1)|^2 dt=O(T), \quad \sigma>\sigma_0,
\end{align}
for almost all $\omega_1\in\Omega_1$.    This is actually a special case of Lemma~10 of
\cite{AL-96-2}.
The cor\-res\-pon\-ding mean value result for $\zeta(s,\alpha,\omega_2;\gb)$ has been shown in
\cite{AJ-AL-06}.
Using those mean value results, we can show \eqref{rk-lem-3-2} in the same way as the proof of
Lemma 4 of \cite{EB-AL-15rj}.
\end{proof}

Lemma \ref{rk-lem-3} together with the weak convergence of of the measures $P_{N,n}$ and $\widehat{P}_{N,n}$ (Lem\-ma~\ref{rk-lem-2}) enables us to prove that the probability measure $P_N$ and one more probability measure defined as
$$
{\widehat P}_N(A)=\frac{1}{N+1}\# \big\{0 \leq k \leq N:{\underline Z}(\underline{s}+ikh,\omega)\in A \big\},\quad A \in {\mathcal{B}}(\underline{H}),
$$
both converge weakly to the same probability measure $P$, i.e., the following satement holds.

\begin{lemma}\label{rk-lem-4}
Suppose that the set $L(\pp,\alpha,h)$ is linearly independent over $\qq$. Then, on $(\underline{H},{\mathcal B}(\underline{H}))$, there exists a probability measure $P$ such that the measures $P_N$ and ${\widehat P}_N$ both converge weakly to $P$ as $N \to \infty$.
\end{lemma}

\begin{proof}
This lemma can be shown analogously to Lemma~5 from \cite{EB-AL-15rj}.
\end{proof}

\begin{proof}[Proof of Theorem~\ref{rk-th-3}]
As usual, in the last step of the proof of the functional discrete limit theorem, we show that the limit measure $P$ in Lemma~\ref{rk-lem-4} coincides with $P_{\underline Z}$.

Define the measurable measure-preserving transformation $\Phi_h: \Omega \to \Omega$ on
the group $\Omega$
by $\Phi_h(\omega)=f_h\omega$, $\omega \in \Omega$, where
$f_h=\{(p^{-ih}: p \in \pp), ((m+\alpha)^{-ih}: m \in \nn_0)\}$.    Again using \eqref{rk-eq-6}, we see that $\{\Phi_h(s)\}$ is a one-parameter group, and is ergodic.
This together with the well-known Birkhoff-Khintchine theorem (see \cite{HC-ML-67}) and the weak convergence of ${\widehat P}_N(A)$ gives that $P(A)=P_{\underline{Z}}(A)$ for all $A \in {\mathcal B}(\underline{H})$.    For the details, consult the proof of Theorem~7 of \cite{EB-AL-15rj}
or Theorem 6.1 of \cite{AL-RM-09}.
\end{proof}

\section{The support of the measure $P_{\underline Z}$}\label{sect-3}

To introduce the support of $P_{\underline Z}$, we repeat the arguments of Section~4 from \cite{RK-KM-15}.

Let $\varphi\in\widetilde{S}$, and $K_1$, $K_2$, $f_1$ and $f_2$ be as in the statement of Theorem~\ref{rk-th-2}.
Then we can find a real number $\sigma_0$ with $\sigma^*<\sigma_0<1$ and a positive
number $M>0$, such that $K_1$ is included in the open rectangle
$$
D_M=\{s: \;\sigma_0<\sigma<1,\; |t|<M\}.
$$
Since $\varphi(s)\in\st$, the pole of $\varphi$ is at most at $s=1$, then,
 in this case, we find that
$$
D_{\varphi}=\{s:\;\sigma>\sigma_0, \;\sigma\neq 1\}.
$$
Therefore $D_M$ is an open subset of $D_{\varphi}$.
Also we can find $T>0$ such that $K_2$ belongs to the open rectangle
$$
D_T=\bigg\{s: \;\frac{1}{2}<\sigma<1, \;|t|<T\bigg\}.
$$

To obtain the support of the measure $P_{\underline{Z}}$ we will use  Theorem~\ref{rk-th-3} with  $D_1=D_M$ and $D_2=D_T$.
Let $S_{\varphi}$ be the set of all $f\in H(D_M)$ which is non-vanishing on $D_M$, or
constantly equivalent to $0$ on $D_M$.

\begin{theorem}\label{rk-th-4}
Suppose that the set $L(\pp,\alpha,h)$ is linearly independent over $\qq$. The support of the measure $P_{\underline{Z}}$ is the set
$S=S_{\varphi}\times H(D_T)$.
\end{theorem}

\begin{proof}
This is an analogue to Lemma~4.3 of \cite{RK-KM-15} or Theorem~8 from \cite{EB-AL-15rj}.
The fact that $\varphi\in\widetilde{S}$ is essentially used here.
\end{proof}

\section{Proof of the mixed joint discrete universality theorem}\label{sect-4}

The proof of Theorem~\ref{rk-th-2} follows from Theorems~\ref{rk-th-3} and \ref{rk-th-4} and the Mergelyan theorem (see \cite{SNM-52}) which we state as a lemma.

\begin{lemma}[Mergelyan]\label{rk-lem-5}
Let $K \subset \cc$ be a compact subset with connected complement, and $f(s)$ be a continuous function on $K$ which is analytic in the interior of $K$. Then, for every $\varepsilon>0$, there exists a polynomial $p(s)$ such that
$$
\sup\limits_{s \in K}|f(s)-p(s)|<\varepsilon.
$$
\end{lemma}

\begin{proof}[Proof of Theorem~\ref{rk-th-2}]
By Lemma~\ref{rk-lem-5}, there exist polynomials $p_1(s)$ and $p_2(s)$ such that
\begin{equation}\label{rk-eq-10}
\sup\limits_{s \in K_1}\big|f_1(s)-e^{p_1(s)}\big|<\frac{\varepsilon}{2}
\end{equation}
and
\begin{equation}\label{rk-eq-11}
\sup\limits_{s \in K_2}\big|f_2(s)-p_2(s)\big|<\frac{\varepsilon}{2}.
\end{equation}

We introduce the set
$$
G=\bigg\{
(g_1,g_2)\in \underline{H}: \sup\limits_{s\in K_1}|g_1(s)-e^{p_1(s)}|<\frac{\varepsilon}{2},
\sup\limits_{s\in K_2}|g_2(s)-p_2(s)|<\frac{\varepsilon}{2}
\bigg\}.
$$
Then $G$ is an open set of the space $\underline{H}$. In virtue of Theorem~\ref{rk-th-4}, it is an open neighbourhood of the element $(e^{p_1(s)},p_2(s))$ of the support of $P_{\underline Z}$. Thus $P_{\underline Z}(G)>0$. Using Theorem~\ref{rk-th-3} and an equivalent statement of the weak convergence in terms of open sets (see \cite{PB-68}), we obtain
$$
\liminf\limits_{N \to \infty}P_N(G)\geq P_{\underline Z}(G)>0.
$$
This and the definitions of $P_N$ and $G$ show that
\begin{eqnarray}\label{rk-eq-12}
\liminf\limits_{N \to \infty}\frac{1}{N+1}\#
\bigg\{
0 \leq k \leq N: && \sup\limits_{s\in K_1}\big|\varphi(s+ikh)-e^{p_1(s)}\big|<\frac{\varepsilon}{2},\cr
&& \sup\limits_{s\in K_2}\big|\zeta(s+ikh,\alpha;\gb)-p_2(s)\big|<\frac{\varepsilon}{2}
\bigg\}>0.
\end{eqnarray}

From \eqref{rk-eq-10} and \eqref{rk-eq-11}, we deduce that
\begin{eqnarray*}
&& \bigg\{
0 \leq k \leq N:  \sup\limits_{s\in K_1}\big|\varphi(s+ikh)-f_1(s)\big|<\varepsilon, \sup\limits_{s\in K_2}\big|\zeta(s+ikh,\alpha;\gb)-f_2(s)\big|<\varepsilon
\bigg\} \cr&&  \supset
\bigg\{
0 \leq k \leq N:  \sup\limits_{s\in K_1}\big|\varphi(s+ikh)-e^{p_1(s)}\big|<\frac{\varepsilon}{2}, \sup\limits_{s\in K_2}\big|\zeta(s+ikh,\alpha;\gb)-p_2(s)\big|<\frac{\varepsilon}{2}
\bigg\}.
\end{eqnarray*}
This together with the inequality \eqref{rk-eq-12} gives the assertion of the theorem.
\end{proof}

\section{The case of modified zeta-functions}\label{sect-5}

In Section \ref{sect-1}, we mentioned an incomplete point in \cite{RK-09}.    An inaccuracy is
actually included in a former paper \cite{RK-DK-09}, whose result is applied to \cite{RK-09}.
On p.~103 of \cite{RK-DK-09}, the same as \eqref{rk-eq-6}
for $(\underline{k},\underline{l})\not =(\underline{0},\underline{0})$
is claimed under the assumption that
$\alpha$ is transcendental and $\exp\{\frac{2\pi}{h}\}$ is rational.   The same reasoning as in the case of
\eqref{rk-eq-6} is valid if there is some $l_m\neq 0$, because from \eqref{rk-eq-6-2} we have
\begin{align}\label{5-1}
\prod_{p\in\pp}p^{k_p}\prod_{m\in\nn_0}(m+\alpha)^{l_m}=\bigg(\exp\bigg\{\frac{2\pi}{h}\bigg\}\bigg)^a,
\end{align}
which contradicts the assumption.   But if all $l_m=0$, then \eqref{5-1} does not produce a contradiction.
Therefore the results in \cite{RK-DK-09}, and hence in \cite{RK-09}, is to be amended.

Write $\exp\big\{\frac{2\pi}{h}\big\}=\frac{a}{b}$, $a,b\in\mathbb{Z}$, $(a,b)=1$, and denote by $\mathbb{P}_h$ the set of
all primes ap\-pea\-ring as prime divisors of $a$ or $b$.  Instead of $Q_N(A)$ defined in
Section \ref{sect-2}, we define $Q_{N,h}(A)$ by replacing $\mathbb{P}$ in the definition of
$Q_N(A)$ by $\mathbb{P}\setminus\mathbb{P}_h$.
Let
$$
\Omega_{1h}=\prod_{p\in\mathbb{P}\setminus\mathbb{P}_h}\gamma_p,
$$
and denote the probability Haar measure on $(\Omega_{1h},\mathcal{B}(\Omega_{1h}))$ by $m_{1hH}$.

\begin{lemma}\label{lem-5-1}
Assume that $\alpha$ is transcendental and $\exp\big\{\frac{2\pi}{h}\big\}$ is rational.    Then $Q_{N,h}$ converges
weakly to the Haar measure $m_{hH}=m_{1hH}\times m_{2H}$ on the space
$\Omega_h=\Omega_{1h}\times\Omega_2$ as
$N\to\infty$.
\end{lemma}

\begin{proof}
If we replace $\mathbb{P}$ by $\mathbb{P}\setminus\mathbb{P}_h$ in \eqref{5-1}, then the resulting equality is impossible even if all $l_m=0$.    Therefore \eqref{rk-eq-6} is valid for any
$(\underline{k},\underline{l})\not =(\underline{0},\underline{0})$, and so we can mimic the proof of
Lemma \ref{rk-lem-1}.
\end{proof}

This lemma is the corrected version of Lemma 2.1 of \cite{RK-DK-09}.
Let $\chi$ be a Dirichlet cha\-rac\-ter, and define a modified Dirichlet $L$-function by
$$
L_h(s,\chi)=\prod_{p\in\mathbb{P}\setminus\mathbb{P}_h}\left(1-\frac{\chi(p)}{p^s}\right)^{-1}.
$$
Then, using Lemma \ref{lem-5-1}, we can show a mixed joint discrete universality theorem
for $L_h(s,\chi)$ and a periodic Hurwitz zeta-function, by the argument in \cite{RK-09}.
This is the corrected version of Theorem 1.7 of \cite{RK-09}, which was already mentioned in
\cite{KM-15}.

It is possible to generalize the above arguments to the class of Matsumoto zeta-func\-tions.
We conclude the present paper with the statement of such results.

Define the modified Matsumoto zeta-function by
\begin{align}\label{5-2}
{\widetilde\varphi}_h(s)=\prod_{m\in\mathbb{N}\setminus\mathbb{N}_h}\prod_{j=1}^{g(m)}\left(1-a_m^{(j)}p_m^{-sf(j,m)}\right)^{-1},
\end{align}
where $\mathbb{N}_h$ is the set of all $m\in\mathbb{N}$ such that $p_m\in\mathbb{P}_h$,
and $\varphi_h(s)={\widetilde\varphi}_h(s+\alpha+\beta)$.
The difference between $\varphi_h(s)$ and $\varphi(s)$ is only finitely many Euler factors,
so their analytic properties are not so different.    In particular, if $\varphi(s)$ satisfies
the properties (a), (b) and (c), then so is $\varphi_h(s)$, too.    Therefore,
the method developed in the previous sections of the present paper can be applied to
$\varphi_h(s)$.   Let
\begin{align*}
\underline{Z}_h(\underline{s})&=(\varphi_h(s_1),\zeta(s_2,\alpha;\gb)),\\
\underline{Z}_h(\underline{s},\omega_h)&=(\varphi_h(s_1,\omega_{1h}),\zeta(s_2,\alpha,\omega_2;\gb)),
\end{align*}
where $\omega_h=(\omega_{1h},\omega_2)\in \Omega_h$.
Define $P_{\underline{Z},h}$ and $P_{N,h}$ analogously to $P_{\underline{Z}}$ and $P_N$, just
replacing $\underline{Z}(\underline{s},\omega)$ and
$\underline{Z}(\underline{s}+ikh)$ by $\underline{Z}_h(\underline{s},\omega_h)$
and $\underline{Z}_h(\underline{s}+ikh)$, respectively.    Then, using Lemma \ref{lem-5-1},
we obtain

\begin{theorem}\label{thm-5-1}
Suppose that $\alpha$ is transcendental and $\exp\big\{\frac{2\pi}{h}\big\}$ is rational.   Then $P_{N,h}$
converges weakly to $P_{\underline{Z},h}$ as $N\to\infty$.
\end{theorem}

\begin{theorem}\label{thm-5-2}
Let $\varphi(s)\in {\widetilde S}$, $K_1$, $K_2$, $f_1(s)$ and $f_2(s)$ satisfy the conditions as in Theorem~\ref{rk-th-1}. Suppose that $\alpha$ is transcendental and $\exp\big\{\frac{2\pi}{h}\big\}$ is rational.
Then, for every $\varepsilon>0$,
\begin{eqnarray*}
\liminf\limits_{N \to \infty}
\frac{1}{N+1}
\#
\bigg\{0\leq k \leq N:
 && \sup\limits_{s \in K_1}|\varphi_h(s+ikh)-f_1(s)|<\varepsilon, \\ &&  \sup\limits_{s\in K_2}|\zeta(s+ikh,\alpha;\gb)-f_2(s)|<\varepsilon\bigg\}>0.
\end{eqnarray*}
\end{theorem}


\end{document}